\documentclass[10pt]{article}
\font\smallit=cmti10
\font\smalltt=cmtt10

\usepackage{amssymb,latexsym,amsmath,epsfig,amsthm} 

\makeatletter

\renewcommand\section{\@startsection {section}{1}{\z@}
{-30pt \@plus -1ex \@minus -.2ex}
{2.3ex \@plus.2ex}
{\normalfont\normalsize\bfseries\boldmath}}

\renewcommand\subsection{\@startsection{subsection}{2}{\z@}
{-3.25ex\@plus -1ex \@minus -.2ex}
{1.5ex \@plus .2ex}
{\normalfont\normalsize\bfseries\boldmath}}

\renewcommand{\@seccntformat}[1]{\csname the#1\endcsname. }

\makeatother

\newtheorem{theorem}{Theorem}

\theoremstyle{definition}
\newtheorem{definition}{Definition}

\newtheorem{example}{Example}

\begin{document}

\begin{center}
\uppercase{\bf \boldmath Abundance of arithmetic progressions in $\mathcal{CR}$-sets}
\vskip 20pt
{\bf Dibyendu De\footnote{Authour is supported by NBHM 02011-6-2021-R\&D-II.}}\\
{\smallit  Department of Mathematics, University of Kalyani,
Kalyani-741235, Nadia, West Bengal, India}\\
{\tt dibyendude@klyuniv.ac.in}\\ 
\vskip 10pt
{\bf Pintu Debnath\footnote{Corresponding author.}}\\
{\smallit Department of Mathematics, Basirhat College,
Basirhat -743412, North 24th parganas, West Bengal, India}\\
{\tt pintumath1989@gmail.com}\\ 

\end{center}
\vskip 20pt
\centerline{\smallit Received: , Revised: , Accepted: , Published: } 
\vskip 30pt


\centerline{\bf Abstract}
\noindent
H.Furstenberg and E.Glasner proved that for an arbitrary
$k\in\mathbb{N}$, any piecewise syndetic set of integers contains a 
$k$-term arithmetic progression and the collection of such progressions
is itself piecewise syndetic in $\mathbb{Z}.$ The above result was
extended for arbitrary semigroups by V. Bergelson and N. Hindman,
using the algebra of the  Stone-\v{C}ech compactification of discrete
semigroups. However,  they provided an abundance for various types
of large sets. In \cite{DHS}, the first author, Neil Hindman and
Dona Strauss introduced two notions of large sets, namely, $J$-set
and $C$-set. In \cite{BG},  V. Bergelson and D. Glasscock introduced
another notion of largeness, which is analogous to the notion of $J$-set,
namely $\mathcal{CR}$- set. All these sets contain arithmetic
progressions of arbitrary length. In \cite{DG}, the second author
and S. Goswami proved that for any $J$-set, $A\subseteq\mathbb{N}$,
the collection $\{(a,b):\,\{a,a+b,a+2b,\ldots,a+lb\}\subset A\}$
is a $J$-set in $(\mathbb{N\times\mathbb{N}},+)$. In this article, 
we prove the same for $\mathcal{CR}$-sets.

\pagestyle{myheadings}
\markright{\smalltt INTEGERS: 24 (2024)\hfill}
\thispagestyle{empty}
\baselineskip=12.875pt
\vskip 30pt

\section{Introduction}

For a  general commutative semigroup $(S,+)$, a set $A\subseteq S$
is said to be syndetic in $(S,+)$, if there exists a finite set $F\subset S$
such that $\bigcup_{t\in F}-t+A=S$. A set $A\subseteq S$ is said
to be thick if for every finite set $E\subset S$, there exists an
element $x\in S$ such that $E+x\subset A$. A set $A\subseteq S$
is said to be a piecewise syndetic set if there exists a finite set $F\subset S$
such that $\bigcup_{t\in F}-t+A$ is thick in $S$ \cite[Definition 4.38, page 101]{HS12}.
It can be proved that a piecewise syndetic set is the intersection
of thick set and syndetic set \cite[Theorem 4.49, page 105]{HS12}.

One of the famous Ramsey theoretic results is the so called van
der Waerden's Theorem \cite{W}, which states that
at least one cell of any partition $\{C_{1},C_{2},\ldots,C_{r}\}$
of $\mathbb{N}$,  contains an arithmetic progression of arbitrary length.
Since arithmetic progressions are invariant under shifts, it follows
that every piecewise syndetic set contains arbitrarily long arithmetic
progressions. The following theorem was proved algebraically by H.
Furstenberg and E. Glasner in \cite{FG} and combinatorially by
Beigelb\"{o}ck in \cite{B}.

\begin{theorem}
Let $k\in\mathbb{N}$ and assume that $S\subseteq\mathbb{Z}$
is piecewise syndetic. Then $\{(a,d)\,:\,\left\{ a,a+d,\ldots,a+kd\right\} \subset S\}$
is piecewise syndetic in $\mathbb{Z}^{2}$.
\end{theorem}

To state the next theorem we need the following Definition. 
\begin{definition}

Let $\left(S,+\right)$ be a commutative semigroup and let
$A\subseteq S$. A is a $J$-set if and only if for every $F\in\mathcal{P}_{f}\left(S^{\mathbb{N}}\right)$,
there exist $a\in S$ and $H\in\mathcal{P}_{f}\left(\mathbb{N}\right)$
such that for each $f\in F$, $a+\sum_{n\in H}f(n)\in A$.
\end{definition}

In \cite{DG}, the second author and S. Goswami proved
the following:

\begin{theorem}
Let $k\in\mathbb{N}$ and assume that $S\subseteq\mathbb{N}$
be a $J$-set. Then $$\{(a,d):\,\left\{ a,a+d,\ldots,a+kd\right\} \subset S\}$$
is also a $J$-set in $\mathbb{N} \times \mathbb{N}$.
\end{theorem}

To express the main result of this article, we have to first
define combinatorially rich set or $\mathcal{CR}$-set introduced
by V. Bergelson and D. Glasscock. For $n,r\in\mathbb{N}$, denote
by $S^{r\times n}$ the set of $r\times n$ matrices with elements
in $S$. For $M=\left(M_{ij}\right)\in S^{r\times n}$ and a non-empty
$\alpha\subseteq\{1,2,\cdots,n\}$, $M_{\alpha j}$ denotes the sum
$\sum_{i\in\alpha}M_{ij}$.

\begin{definition}
Let $\left(S,+\right)$ be a commutative semigroup. A subset
$A\subseteq S$ is a $\mathcal{CR}$-set if for all $n\in\mathbb{N}$,
there exists a $r\in\mathbb{N}$ such that for all $M\in S^{r\times n}$,
there exists a non-empty set $\alpha\subseteq\{1,2,\ldots,r\}$, and
$s\in S$ such that for all $j\in\{1,2,\ldots,n\}$, 
\[
s+M_{\alpha,j}\in A.
\]
We denote by $\mathcal{CR}\left(S,+\right)$, the class of combinatorially
rich subsets of $\left(S,+\right)$.
\end{definition}

Choosing $\left(S,+\right)=\left(\mathbb{N},+\right)$ and
 $M_{ij}=j$, from the above definition, there exist a nonempty
$\alpha\subseteq\{1,2,\ldots,r\}$ and $s\in\mathbb{N}$ such that
$$\left\{ s+\mid\alpha\mid,s+2\mid\alpha\mid,\ldots,s+n\mid\alpha\mid\right\} \subset A .$$
Thus, combinatorially rich sets in $(\mathbb{N},+)$ are $AP$-rich i.e., they
contain arbitrary long arithmetic progressions. It can be stated that piecewise syndetic subsets of $S$  are $\mathcal{CR}$-sets and that  $\mathcal{CR}$-subsets of $S$  are $J$-sets. Of course, the first inclusion implies that the set of $\mathcal{CR}$-subsets of $S$ is non-empty, and the second inclusion immediately implies that $CR$-subsets of $S$ contain arbitrarily long arithmetic progressions. We prove that for any  $\mathcal{CR}$-set
$A\subseteq\mathbb{N}$, the collection $\{(a,b):\,\{a,a+b,a+2b,\ldots,a+lb\}\subset A\}$
is a $\mathcal{CR}$-set in $(\mathbb{N\times\mathbb{N}},+)$ and the same
result for  essential $\mathcal{CR}$-sets. The next section is devoted
to essential $\mathcal{CR}$-sets. 

\section{Essential $\mathcal{CR}$-set}

A collection $\mathcal{F\subseteq P}\left(S\right)\setminus\left\{ \emptyset\right\} $
is upward hereditary if whenever $A\in\mathcal{F}$ and $A\subseteq B\subseteq S$
then it follows that $B\in\mathcal{F}$. A nonempty and upward hereditary
collection $\mathcal{F\subseteq P}\left(S\right)\setminus\left\{ \emptyset\right\} $
will be called a family. If $\mathcal{F}$ is a family, the dual family
$\mathcal{F}^{*}$ is given by
$$
\mathcal{F}^{*}=\{E\subseteq S:\forall A\in\mathcal{F},E\cap A\neq\emptyset\}.
$$
A family $\mathcal{F}$ possesses the Ramsey property if,  whenever
$A\in\mathcal{F}$ and $A=A_{1}\cup A_{2}$ there is some $i\in\left\{ 1,2\right\} $
such that $A_{i}\in\mathcal{F}$.

We give a brief review of the algebraic structure of the Stone-\v{C}ech
compactification of discrete semigroups.

Let $S$ be a discrete semigroup.The elements of $\beta S$ are regarded as ultrafilters on $S$. Let $\overline{A}=\left\{p\in \beta S:A\in p\right\}$. The set $\{\overline{A}:A\subset S\}$ is a basis for the closed sets
of $\beta S$. The operation `$\cdot$' on $S$ can be extended to
the Stone-\v{C}ech compactification $\beta S$ of $S$ so that $(\beta S,\cdot)$
is a compact right topological semigroup (meaning that for each    $p\in\beta$ S the function $\rho_{p}\left(q\right):\beta S\rightarrow\beta S$ defined by $\rho_{p}\left(q\right)=q\cdot p$ 
is continuous) with $S$ contained in its topological center (meaning
that for any $x\in S$, the function $\lambda_{x}:\beta S\rightarrow\beta S$
defined by $\lambda_{x}(q)=x\cdot q$ is continuous). This is a famous
Theorem due to Ellis that if $S$ is a compact right topological semigroup
then the set of idempotents $E\left(S\right)\neq\emptyset$. A nonempty
subset $I$ of a semigroup $T$ is called a $\textit{left ideal}$
of $S$ if $TI\subset I$, a $\textit{right ideal}$ if $IT\subset I$,
and a $\textit{two sided ideal}$ (or simply an $\textit{ideal}$)
if it is both a left and right ideals. A minimal left ideal
is the left ideal that does not contain any proper left ideal. Similarly,
we can define minimal right ideal and the smallest ideal.

Any compact Hausdorff right topological semigroup $T$ has the smallest
two sided ideal

$$
\begin{aligned}
	K(T) & =  \bigcup\{L:L\text{ is a minimal left ideal of }T\}\\
	&=  \bigcup\{R:R\text{ is a minimal right ideal of }T\}.
\end{aligned}$$

Given a minimal left ideal $L$ and a minimal right ideal $R$, $L\cap R$
is a group, and in particular contains an idempotent. If $p$ and
$q$ are idempotents in $T$; we write $p\leq q$ if and only if $pq=qp=p$.
An idempotent is minimal with respect to this relation if and only
if it is a member of the smallest ideal $K(T)$ of $T$. Given $p,q\in\beta S$
and $A\subseteq S$, $A\in p\cdot q$, if and only if the set $\{x\in S:x^{-1}A\in q\}\in p$,
where $x^{-1}A=\{y\in S:x\cdot y\in A\}$. See \cite{HS12} for
an elementary introduction to the algebra of $\beta S$ and for any
unfamiliar details.

It is known that the family $\mathcal{F}$ has the Ramsey
property iff the family $\mathcal{F}^{*}$ is a filter. For a family
$\mathcal{F}$ with the Ramsey property, let $\beta(\mathcal{F})=\{p\in\beta S:p\subseteq\mathcal{F}\}$.
Then we get the following from \cite[Theorem 5.1.1]{C}:

\begin{theorem}
Let $S$ be a discrete set. For every family $\mathcal{F\subseteq P}\left(S\right)$
with the Ramsey property, $\beta\left(\mathcal{F}\right)\subseteq\beta S$
is closed. Furthermore, $\mathcal{F}=\cup\beta\left(\mathcal{F}\right)$.
Also if $K\subseteq\beta S$ is closed, $\mathcal{F}_{K}=\left\{ E\subseteq S:\overline{E}\cap K\neq\emptyset\right\} $
is a family with the Ramsey property and $\overline{K}=\beta\left(\mathcal{F}_{K}\right)$.
\end{theorem}

Let $S$ be a discrete semigroup, then for every family,  $\mathcal{F\subseteq P}\left(S\right)$
with the Ramsey property, $\beta\left(\mathcal{F}\right)\subseteq\beta S$
is closed. If $\beta\left(\mathcal{F}\right)$ be a subsemigroup of
$\beta S$, then $E\left(\beta\mathcal{F}\right)\neq\emptyset$.

\begin{definition}
Let $\mathcal{F}$ be a family with the Ramsay property such
that $\beta(\mathcal{F})$ is a subsemigroup of $\beta S$ and $p$
be an idempotent in $\beta(\mathcal{F})$, then each member of $p$
is called an essential $\mathcal{F}$-set. 
\end{definition}

The family $\mathcal{F}$ is called left (right) shift-invariant
if for all $s\in S$ and all $E\in\mathcal{F}$,  one has $sE\in\mathcal{F}(Es\in\mathcal{F})$.
The family $\mathcal{F}$ is called left (right) inverse shift-invariant
if for all $s\in S$ and all $E\in\mathcal{F}$,  one has $s^{-1}E\in\mathcal{F}(Es^{-1}\in\mathcal{F})$.
 We derive the following theorem from \cite[Theorem 5.1.2]{C}:
\begin{theorem}
If $\mathcal{F}$ is a family having the Ramsey property
then $\beta\mathcal{F}\subseteq\beta S$ is a left ideal if and only
if $\mathcal{F}$ is left shift-invariant. Similarly, $\beta\mathcal{F}\subseteq\beta S$
is a right ideal if and only if $\mathcal{F}$ is right shift-invariant.
\end{theorem}

From \cite[Theorem 5.1.10]{C}, we can identify those
families $\mathcal{F}$ with Ramsey property for which $\beta\left(\mathcal{F}\right)$
is a subsemigroup of $\beta S$ . The condition is a rather technical
weakening of left shift invariance.
\begin{theorem}
Let $S$ be any semigroup, and let $\mathcal{F}$ be a family
of subsets of $S$ having the Ramsey property. Then the following
are equivalent:
\begin{itemize}
    \item[(1)]  $\beta\left(\mathcal{F}\right)$ is a subsemigroup of
$\beta S$.
\item[(2)] $\mathcal{F}$ has the following property: 

If $E\subseteq S$
is any set, and if there is $A\in\mathcal{F}$ such that for all finite
$H\subseteq A$,  one has $\left(\cap_{q\in H}x^{-1}E\right)\in\mathcal{F}$,
then $E\in\mathcal{F}$.
\end{itemize}

\end{theorem}

The elementary characterization of essential $\mathcal{F}$-sets
is known from \cite[Theorem 5]{DDG}.

\begin{definition}
Let $\omega$ be the first infinite ordinal and let each ordinal denotes the set of all it's predecessors. In particular, $0=\emptyset,$
for each $n\in\mathbb{N},\:n=\left\{ 0,1,...,n-1\right\} $.
\begin{itemize}
    \item[(a)] If $f$ is a function and $dom\left(f\right)=n\in\omega$,
then for all $x$, $f^{\frown}x=f\cup\left\{ \left(n,x\right)\right\} $.

    \item[(b)] Let $T$ be a set of  functions whose domains are members
of $\omega$. For each $f\in T$, $B_{f}\left(T\right)=\left\{ x:f^{\frown}x\in T\right\}$.
\end{itemize}
\end{definition}

We get the following theorem from \cite[Theorem 5]{DDG}
which plays a vital role in this article.

\begin{theorem}\label{Elementary F sets}
	Let $\left(S,.\right)$ be a semigroup, and assume that $\mathcal{F}$
	is a family of subsets of $S$ with the Ramsay property such that
	$\beta\left(\mathcal{F}\right)$ is a subsemigroup of $\beta S$.
	Let $A\subseteq S$. Then the statements (a), (b) and (c) are equivalent and
	are implied by statement (d). If $S$ is countable, then all the five
	statements are equivalent.
	
	\begin{itemize}
			
\item[(a)] $A$ is an essential $\mathcal{F}$-set.
	
\item[(b)] There is a non empty set $T$ of function such that
\begin{itemize}

	\item[(i)]  for all $f\in T$,$\text{domain}\left(f\right)\in\omega$
	and $rang\left(f\right)\subseteq A$;
	
	\item[(ii)]  for all $f\in T$ and all $x\in B_{f}\left(T\right)$,
	$B_{f^{\frown}x}\subseteq x^{-1}B_{f}\left(T\right)$; and
	
	\item[(iii)]  for all $F\in\mathcal{P}_{f}\left(T\right)$, $\cap_{f\in F}B_{f}(T)$
	is a $\mathcal{F}$-set.
	
	\end{itemize}

\item[(c)]  There is a downward directed family $\left\langle C_{F}\right\rangle _{F\in I}$
of subsets of $A$ such that
\begin{itemize}
	
\item[(i)] for each $F\in I$ and each $x\in C_{F}$ there exists
$G\in I$ with $C_{G}\subseteq x^{-1}C_{F}$ and

\item[(ii)] for each $\mathcal{F}\in\mathcal{P}_{f}\left(I\right),\,\bigcap_{F\in\mathcal{F}}C_{F}$
is a $\mathcal{F}$-set. 

\end{itemize}

\item[(d)] There is a decreasing sequence $\left\langle C_{n}\right\rangle _{n=1}^{\infty}$
of subsets of $A$ such that 

\begin{itemize}

\item[(i)] for each $n\in\mathbb{N}$ and each $x\in C_{n}$, there
exists $m\in\mathbb{N}$ with $C_{m}\subseteq x^{-1}C_{n}$ and 

\item[(ii)] for each $n\in\mathbb{N}$, $C_{n}$ is a $\mathcal{F}$-set.
\end{itemize}

\end{itemize}
\end{theorem}

Let $(S,+)$ be a commutative semigroup,  and by \cite[Lemma 2.14]{BG},
the class $\mathcal{CR}$ is partition regular, and it is trivial that the 
class $\mathcal{CR}$ is translation invariant. Hence $\beta\left(\mathcal{CR}\right)$
is a closed subsemigroup of $\beta\left(S\right)$ by Theorem 4. Let
$A$ be a subset of $(S,+)$. We call $A$  an essential $\mathcal{CR}$-set
iff $A\in p$ for some $ p\in E\left(\beta\left(\mathcal{CR}\right)\right)$. Then
from the above, we get the following theorem:

\begin{theorem}
Let $(S,+)$ be a countable commutative semigroup, then the
following are equivalent:
\begin{itemize}
\item[(a)]  $A$ is an essential $\mathcal{CR}$-set.

\item[(b)] There is a decreasing sequence $\left\langle C_{n}\right\rangle _{n=1}^{\infty}$
of subsets of $A$ such that 
\begin{itemize}
\item[(i)] for each $n\in\mathbb{N}$ and each $x\in C_{n}$, there
exists $m\in\mathbb{N}$ with $C_{m}\subseteq x^{-1}C_{n}$ and 

\item[(ii)]  for each $n\in\mathbb{N}$, $C_{n}$ is a $\mathcal{CR}$-set.
\end{itemize}
\end{itemize}
\end{theorem}

\section{Proof of the main theorem}

Now,  we are going to prove the abundance of arithmetic progressions
in $\mathcal{CR}$-sets.

\begin{definition}
Let $A=\left(a_{ij}\right)_{m\times n_{1}}$ and $B=\left(b_{ij}\right)_{m\times n_{2}}$
be two matrices. The concatenation of two matrices $A$ and $B$ is $C=A^{\frown}B$ defined by 
$C=\left(c_{ij}\right)_{m\times\left(n_{1}+n_{2}\right)}$, where 

 $$c_{ij}=
 \begin{cases}

a_{ij} \text{ if } j\leq n_{1}\\
b_{ij}  \text{ if } j>n_{2}.
\end{cases}$$
\end{definition}

\begin{example} Let 
$A=
\begin{pmatrix}
   3 & 6\\
7 & 4\\
1 & 3 
\end{pmatrix}$,
$B=
\begin{pmatrix}
5 & 8 & 9 & 1\\
6 & 8 & 3 & 5\\
7 & 9 & 2 & 1
\end{pmatrix}$ and $C=\begin{pmatrix}
6\\
9\\
8
\end{pmatrix}$. Then $$A^{\frown}B^{\frown}C=\begin{pmatrix}
3 & 6 & 5 & 8 & 9 & 1 & 6\\
7 & 4 & 6 & 8 & 3 & 5 & 9\\
1 & 3 & 7 & 9 & 2 & 1 & 8
\end{pmatrix}.$$
\end{example}

\begin{theorem}\label{main theoreom1}
Let $\left(S,+\right)$ be a commutative semigroup.  Let $A$ be a $\mathcal{CR}$-set in $S$, and $l\in\mathbb{N}$.
Then the set $$\{(a,b):\,\{a,a+b,a+2b,\ldots,a+lb\}\subset A\}$$ is a
$\mathcal{CR}$-set in $(S\times S,+)$.
\end{theorem}

\begin{proof}
Let $C=\{(a,b):\,\{a,a+b,a+2b,\ldots,a+\left(l-1\right)b\}\subset A\}$.
Since $A$ is $\mathcal{CR}$-set, for any $n\in\mathbb{N}$, we can
find $r$ such that $M\in S^{r\times ln}$, there exist $ a\in S$,
and $\alpha\subset\left\{ 1,2,\ldots,r\right\} $ such that $a+\sum_{i\in\alpha}M_{i,j}\in A$
for all $j\in\left\{ 1,2,\ldots,ln\right\} $. Let $M^{\prime}\in\left(S\times S\right)^{r\times ln}$
then $M^{\prime}=\left(\left(M_{ij}^{1},M_{ij}^{2}\right)\right)_{r\times ln}$. Let $s\in S$ and 

$$ M^{k}=\begin{pmatrix}
M_{11}^{1}+k(s+M_{11}^{2}) & M_{12}^{1}+k(s+M_{12}^{2}) & \cdots & M_{1n}^{1}+k(s+M_{1n}^{2})\\
M_{21}^{1}+k(s+M_{21}^{2}) & M_{22}^{1}+k(s+M_{22}^{2}) & \ldots & M_{2n}^{1}+k(s+M_{2n}^{2})\\
\vdots & \vdots & \ldots & \vdots\\
M_{r1}^{1}+k(s+M_{r1}^{2}) & M_{r2}^{1}+k(s+M_{r2}^{2}) & \ldots & M_{rn}^{1}+k(s+M_{rn}^{2})
\end{pmatrix}$$
 for $k\in\left\{ 0,1,2,\ldots,l-1\right\} $.\\
Now $M=M^{0\frown}M^{1\frown}\ldots^{\frown}M^{l-1}$ and
$M$ is a $r\times ln$ matrix, there exists $\alpha\in \left\{  1,2,\ldots,r\right\}$
such that $a+\sum_{i\in\alpha}M_{i,j}\in A$ for all $j\in\left\{ 1,2,\ldots,ln\right\} $. \\
Which implies that $a+\sum_{i\in\alpha}M_{i,j}^{1}+k(s+M_{ij}^{2})\in A$
for all $j\in\left\{ 1,2,\ldots,n\right\} $ and $k\in\left\{ 0,1,2,\ldots,l-1\right\} $.\\
$\implies a+\sum_{i\in\alpha}M_{i,j}^{1}+k\sum_{i\in\alpha}(s+M_{ij}^{2})\in A$
for all $j\in\left\{ 1,2,\ldots,n\right\} $ 
and $k\in\left\{ 0,1,2,\ldots,l-1\right\} $.\\
$\implies a+\sum_{i\in\alpha}M_{i,j}^{1}+k(\mid\alpha\mid s+\sum_{i\in\alpha}M_{ij}^{2})\in A$
for all $j\in\left\{ 1,2,\ldots,n\right\} $ and $k\in\left\{ 0,1,2,\ldots,l-1\right\} $.\\
$\implies a+\sum_{i\in\alpha}M_{i,j}^{1}+k(\mid\alpha\mid s+\sum_{i\in\alpha}M_{ij}^{2})\in A$
for all $j\in\left\{ 1,2,\ldots,n\right\} $ and $k\in\left\{ 0,1,2,\ldots,l-1\right\} $.\\
$\implies\left(a+\sum_{i\in\alpha}M_{i,j}^{1},\mid\alpha\mid s +\sum_{i\in\alpha}M_{ij}^{2}\right)\in C$
for all $j\in\left\{ 1,2,\ldots,n\right\} $.\\
$\implies\left(a,\mid\alpha\mid s \right)+\sum_{i\in\alpha}\left(M_{i,j}^{1}+M_{i,j}^{2}\right)\in C$
for all $j\in\left\{ 1,2,\ldots,n\right\} $.\\
Hence $C$ is a $\mathcal{CR}$-set.
\end{proof}
Now,  we can prove the abundance of arithmetic progressions
in  an essential $\mathcal{CR}$-set, using elementary characterization
and the above theorem.
\begin{theorem}\label{main theorem 2}
    Let $A$ be an essential $\mathcal{CR}$-set in $\mathbb{N}$,
and $l\in\mathbb{N}$. Then the set $$\{(a,b):\,\{a,a+b,a+2b,\ldots,a+lb\}\subset A\}$$
is an essential $\mathcal{CR}$-set in $(\mathbb{N\times\mathbb{N}},+)$.
\end{theorem}

\begin{proof}
As $A$ is an essential $\mathcal{CR}$- set, there exists a
decreasing sequence of $\mathcal{CR}$-sets in $\mathbb{N}$, $\{A_{n}:n\in\mathbb{N}\}$
satisfying the property b(i) of  the Theorem 7. 
As all $A_{n}$ are $\mathcal{CR}$-sets $\forall n\in\mathbb{N}$
in the following sequence,
$$
A\supseteq A_{1}\supseteq A_{2}\supseteq\ldots\supseteq A_{n}\supseteq\ldots
$$
and for $i\in\mathbb{N}$, $B_{i}=\{(a,b)\in\mathbb{N\times\mathbb{N}}:\,\{a,a+b,a+2b,\ldots,a+\left(l-1\right)b\}\subset A_{i}\}$
are $\mathcal{CR}$-sets in $\mathbb{N}\times\mathbb{N}$.
Consider,
$$
B\supseteq B_{1}\supseteq B_{2}\supseteq\ldots\supseteq B_{n}\supseteq\ldots.
$$
Pick $n\in\mathbb{N}$ and $(a,b)\in B_{n}$. Then
$\{a,a+b,a+2b,\ldots,a+\left(l-1\right)b\}\subset A_{n}$. 
Now,  by property (b)(i) of the Theorem 7, there exists $N_{i}\in\mathbb{N}$
such that $a+ib\in A_{i}$ for $i=0,1,2,\ldots,l-1$ and $A_{N_{i}}\subseteq-(a+ib)+A_{n}$.
 Taking $N=\max\left\{ N_{0},N_{1},\ldots,N_{l-1}\right\} $, we get
$$
A_{N}\subseteq\bigcap_{i=0}^{l-1}\left(-(a+ib)+A_{n}\right).
$$
Now any $\left(a_{1},b_{1}\right)\in B_{N}$ implies $$\left\{a_{1},a_{1}+b_{1},a_{1}+2b_{1},\ldots,a_{1}+\left(l-1\right)b_{1}\right\}\subseteq A_{N}\subseteq\bigcap_{i=0}^{l-1}\left(-(a+ib)+A_{n}\right).$$
So $(a_{1}+a)+i.(b_{1}+b)\in A_{n}$   for all $ i\in\{0,1,2,\ldots,l-1\}$.
Hence $(a_{1},b_{1})\in-(a,b)+B_{n}$.
This implies $B_{N}\subseteq-(a,b)+B_{n}$. 
Therefore, for any $(a,b)\in B_{n}$, there exists $N\in\mathbb{N}$
such that $B_{N}\subseteq-(a,b)+B_{n}$, showing the  property (b)(i) of  the Theorem 7. This proves the theorem.

\end{proof}

\noindent\textbf{Acknowledgements.} The authors would like to convey their heartfelt thanks to  the  referee for giving valuable expertise comments in the previous draft of this article.

\end{document}